\documentclass{amsart}
\usepackage{amscd}
\usepackage{fullpage}
\usepackage{amssymb}
\usepackage[all, knot]{xy}
\usepackage{mathtools}
\usepackage{color}
\xyoption{all}
\xyoption{arc}
\usepackage{hyperref}

\def\ctr{{\mathrm{ctr}}}

\def\del{\partial}

\def\cube#1#2#3#4#5#6#7#8{
& #5 \ar[rr] \ar[dl] \ar@{-}[d] && #6 \ar[dd] \ar[dl] \\
#1 \ar[rr] \ar[dd]  & \ar[d] & #2 \ar[dd] \\
& #7 \ar@{-}[r] \ar[dl] & \ar[r] & #8 \ar[dl] \\
#3 \ar[rr] && #4 \\
}

\def\dg{{\mathrm{dg}}}
\def\coh{\operatorname{coh}}
\def\Cohdg{\on{Coh}_\dg}

\def\smsh{\wedge}

\def\sing{{\text{sing}}}

\def\ker{\operatorname{ker}}
\def\cone{\operatorname{cone}}

\def\cA{\mathcal A}
\def\cB{\mathcal B}
\def\cC{\mathcal C}

\def\cD{\mathcal D}

\def\cK{\mathcal K}

\def\coker{\operatorname{coker}}

\def\RHom{\operatorname{{\mathbb R} Hom}}

\def\sExt{\sExt}

\def\End{\operatorname{End}}

\def\Spec{\operatorname{Spec}}

\def\bu{\cdot}

\def\into{\hookrightarrow}
\def\onto{\twoheadrightarrow}

\def\a{\alpha}
\def\b{\beta}
\def\g{\gamma}

\newcommand{\A}{\mathbb{A}}

\newcommand{\C}{\mathbb{C}}
\newcommand{\Z}{\mathbb{Z}}
\newcommand{\N}{\mathbb{N}}

\numberwithin{equation}{section}

\theoremstyle{plain} 
\newtheorem{thm}[equation]{Theorem}
\newtheorem{thm-conj}[equation]{Theorem-Conjecture}
\newtheorem{defn-conj}[equation]{Definition-Conjecture}

\newtheorem*{introthm*}{Theorem}

\newtheorem{cor}[equation]{Corollary}
\newtheorem{lem}[equation]{Lemma}
\newtheorem{prop}[equation]{Proposition}

\newtheorem{conj}[equation]{Conjecture}

\theoremstyle{definition}
\newtheorem{defn}[equation]{Definition}

\newtheorem{ex}[equation]{Example}

\theoremstyle{remark}
\newtheorem{rem}[equation]{Remark}

\newtheorem{notation}[equation]{Notation}

\newtheorem{terminology}[equation]{Terminology}

\def\D{D}
\def\Perf{\operatorname{Perf}}

\def\coh{\operatorname{coh}}

\newcommand{\Hom}{\operatorname{Hom}}

\newcommand{\xra}[1]{\xrightarrow{#1}}

\newcommand{\id}{\operatorname{id}}

\def\s{\sigma}

\def\tr{\operatorname{tr}}

\def\n{\nabla}

\def\and{ \text{ and } }

\def\can{\mathrm{can}}

\def\op{{\mathrm{op}}}

\def\MC{\operatorname{MC}}

\def\O{\Omega}

\def\on{\operatorname}

\def\bu{\bullet}
\def\wQ{\widetilde{Q}}
\def\wf{\widetilde{f}}

\def\D{\on{D}}
\def\Dbdg{\D^{\on{b}}_{\on{dg}}}
\def\DbdgZ{\D^{\on{b}, Z}_{\on{dg}}}
\def\DbZ{\D^{\on{b}, Z}}

\def\DbdgY{\D^{\on{b}, Y}_{\on{dg}}}
\def\DbdgU{\D^{\on{b}, Z \cap U}_{\on{dg}}(U)}
\def\DbdgV{\D^{\on{b}, Z \cap V}_{\on{dg}}(V)}
\def\DbdgUV{\D^{\on{b}, Z \cap U \cap V}_{\on{dg}}(U \cap V)}
\def\Perfdg{\Perf_{\on{dg}}}
\def\sing{\on{sing}}

\def\dR{\on{dR}}
\def\om{\omega}
\def\fiber{\on{fiber}}
\def\MR#1{}

\title{D\'evissage for periodic cyclic homology of complete intersections}
\author{Michael K. Brown}
\author{Mark E. Walker}
\date{}

\vspace{18mm}  

\thanks{The authors were partially supported by NSF grants DMS-2302373 and DMS-2200732.}

\makeatletter
\@namedef{subjclassname@2020}{%
  \textup{2020} Mathematics Subject Classification}
\makeatother

\begin{document}
\maketitle
\begin{abstract}
We prove that the d\'evissage property holds for periodic cyclic homology for a local complete intersection embedding into a smooth scheme. As a consequence, we show that the
complexified topological Chern character maps for the bounded derived category and singularity category of a local complete intersection are isomorphisms, proving new
cases of the Lattice Conjecture in noncommutative Hodge theory. 
\end{abstract}

\tableofcontents

\section{Introduction}

Given a closed embedding $i : Z \into X$ of noetherian schemes, one has a pushforward functor
$$i_* : \coh(Z) \to \coh^Z(X)$$
 from coherent sheaves on $Z$ to coherent sheaves on $X$ supported in $Z$. While this functor is far from being an equivalence in general, it is a fundamental result of Quillen \cite{quillen} that $i_*$ induces an isomorphism on $G$-theory; that is, $G$-theory has the \emph{d\'evissage property}. In more detail: writing $G_*(Z)$ and $G_*^Z(X)$ for the algebraic $K$-theory groups of $\coh(Z)$ and $\coh^Z(X)$, the induced map
$
i_* : G_*(Z) \to G_*^Z(X)
$
is an isomorphism. One can recast this result in the language of differential graded (dg) categories in the following way: given appropriate dg-enhancements $\Dbdg(Z)$ and $\DbdgZ(X)$ of the bounded derived categories of $\coh(Z)$ and $\coh^Z(X)$, the induced map
$
i_* : K_*(\Dbdg(Z)) \to K_*(\DbdgZ(X))
$
on algebraic $K$-theory is an isomorphism. 

Versions of the d\'evissage property are now known to be enjoyed by a host of invariants. For instance, a result of Tabuada-Van den Bergh \cite[Theorem 1.8]{TVdB} states that, given a closed immersion $i : Z \to X$ of \emph{smooth} schemes over a field $k$ and any localizing $\A^1$-homotopy invariant $E$ of dg-categories, the map
\begin{equation}
\label{E}
i_* : E_*(\Dbdg(Z)) \to E_*(\DbdgZ(X))
\end{equation}
is an isomorphism. In particular, \eqref{E} is an isomorphism when $E = HP$, the periodic cyclic homology functor.\footnote{This special case admits a more elementary proof via the
  Hochschild-Kostant-Rosenberg formula; see Remark \ref{tabuada}.} The main goal of this paper is to establish the d\'evissage property for periodic cyclic homology in the case of
an embedding of a \emph{complete intersection} into a smooth scheme, with a view toward proving new cases of the Lattice Conjecture in noncommutative Hodge theory.

To state our result precisely, we make the following definition: a closed embedding $Z \into X$ of noetherian schemes is a \emph{local complete intersection}, or \emph{lci}, if 
there is an affine open cover $U_i = \Spec(Q_i)$ of $X$ such that each $U_i \cap Z$ is equal to $\Spec(Q_i / I_i)$ for some ideal $I_i \subseteq Q_i$ that is generated by a regular
sequence. The following is our main result: 

\begin{thm}
\label{devissage}
Let $k$ be a field of characteristic $0$, $X$ a smooth scheme over $k$, and $i : Z \into X$ an lci closed embedding.
The map $i_*: \Dbdg(Z) \to \DbdgZ(X)$ induces a quasi-isomorphism on periodic cyclic complexes:
$$
HP(\Dbdg(Z)) \xra{\simeq} HP(\DbdgZ(X)).
$$
\end{thm}

In fact, Theorem \ref{devissage} can be extended slightly; see Corollary \ref{generalization}. The $\on{char}(k) = 0$ assumption is necessary to invoke a version of the
Hochschild-Kostant-Rosenberg 
Theorem (see Theorem~\ref{HKR}) and also Lemma~\ref{philemma}. To prove Theorem~\ref{devissage}, we use a Mayer-Vietoris argument to reduce to the affine case. We then proceed via an explicit
calculation using versions of Koszul duality and the Hochschild-Kostant-Rosenberg formula involving matrix factorization categories. While Theorem~\ref{devissage} extends the
aforementioned result \cite[Theorem 1.8]{TVdB} of Tabuada-Van den Bergh, their argument does not adapt to our setting, and so our proof is completely different from theirs; see
Remark~\ref{tabuada} for details.

\subsection*{Applications}

As a first application, we apply Theorem~\ref{devissage} to prove new cases of Blanc's Lattice Conjecture:

\begin{conj}[\cite{blanc} Conjecture 1.7]
\label{latticeconj}
Given a smooth and proper $\C$-linear dg-category $\cC$, the topological Chern character map $K_*^{\on{top}}(\cC) \otimes_\Z \C \to HP_*(\cC)$ is an isomorphism. 
\end{conj}

The motivation for the Lattice Conjecture is that topological $K$-theory is believed to provide the rational lattice in the (conjectural) noncommutative Hodge structure on the periodic cyclic homology of any smooth and proper dg-category, in the framework of Katzarkov-Kontsevich-Pantev's noncommutative Hodge theory~\cite{KKP}. While Conjecture \ref{latticeconj} involves smooth and proper dg-categories, it is known to hold in many cases beyond this purview. More precisely, the Lattice Conjecture is known for the following dg-categories; in what follows, $\Perfdg( - )$ denotes the dg-category of perfect complexes on $(- )$.
\begin{enumerate}
\item $\Perfdg(X)$, where $X$ is a separated, finite type scheme over $\C$ \cite{blanc} (see also \cite[Theorem 1.1]{kon} for a generalization of this result to derived schemes);
\item $\Perfdg(X)$, for $X$ a smooth Deligne-Mumford stack \cite[Corollary 2.19]{HLP};
\item a connected, proper dg-algebra $A$ \cite[Theorem 1.1]{kon};
\item a connected dg-algebra $A$ such that $H_0(A)$ is a nilpotent extension of a commutative $\C$-algebra of finite type \cite[Theorem 1.1]{kon}.
\end{enumerate}

We prove the following:

\begin{thm}
\label{latticeci}
Let $X$ be a noetherian $\C$-scheme such that every point has an open neighborhood that admits an lci closed embedding into a smooth $\C$-scheme. The Lattice Conjecture holds for both the dg-bounded derived category $\Dbdg(X)$ and the dg-singularity category $\D^{\sing}_{\on{dg}}(X)$ of $X$. 
\end{thm}


Theorem \ref{latticeci} opens the door to studying noncommutative Hodge structures (in the sense of \cite{KKP}) of singularity categories of complete intersections, building on the robust literature on Hodge-theoretic properties of such singularity categories \cite{BD, BFK1, BFK2, BW, BW1,BW2, CT, dyckerhoff, efimov2,  HLP, KP, kim1,PV,  segal, shk2, shklyarov1}. We will explore this in detail in the case of singularity categories of hypersurfaces in a forthcoming paper.

\medskip
As a second application, we use the direct calculations in our proof of Theorem~\ref{devissage} to explicitly compute the boundary map in a certain localization sequence on periodic cyclic homology. In more detail: let $Q$ be an essentially smooth $k$-algebra, $f \in Q$ a non-zero-divisor, and $R = Q / f$. Since periodic cyclic homology is a localizing invariant of dg-categories, Theorem~\ref{devissage} implies that there is a long exact sequence
$$
\cdots \to HP_j(\Dbdg(R)) \to HP_j(Q) \to HP_j(Q[1/f]) \xra{\del} HP_{j-1}(\Dbdg(R)) \to \cdots.
$$
We give an explicit formula for the boundary map $\del_j$ in this sequence: see Theorem~\ref{thm:boundarymap}. This formula plays a key role in our aforementioned forthcoming work on noncommutative Hodge structures for singularity categories and was a main source of motivation for the present paper.

\begin{rem}
It was brought to our attention by Adeel Khan after the first version of this paper was posted that Theorem~\ref{devissage} may also be obtained as an application of results of Preygel in \cite{preygel}; see \cite[Appendix A]{khan}. In fact, Khan subsequently used this idea to obtain a more general version of Theorem~\ref{devissage} involving algebraic spaces and without the lci assumption~\cite[Theorem A.2]{khan}. This also leads to a more general version of Theorem~\ref{latticeci}: see \cite[Theorem B]{khan}. However, the explicit calculations in our proof of Theorem~\ref{devissage} are still crucial, for instance, to our proof of the formula for the boundary map in the long exact sequence above (see Theorem~\ref{thm:boundarymap}); we do not see a direct way to compute this boundary map using Preygel's results.
\end{rem}

\subsection*{Acknowledgements} We thank Gon\c calo Tabuada for a helpful correspondence during the preparation of this article, and we also thank Adeel Khan for alerting us to the alternative proof of Theorem~\ref{devissage} using work of Preygel in \cite{preygel}, which now appears in his paper \cite{khan}. Finally, we thank the referee for their careful reading and helpful suggestions.

\section{Notation and background}
\subsection{Notation} Let $k$ be a characteristic 0 field. We index cohomologically unless otherwise noted. 
\subsubsection{Dg-enhancements of derived categories}
\label{enhancement}
Given a noetherian $k$-scheme $X$, let $\Cohdg^b(X)$ denote the dg-category of bounded complexes of coherent sheaves on $X$. (A technical set-theoretic point: we are implicitly
considering all categories of modules or sheaves---and complexes thereof---that arise in this paper in a fixed Grothendieck universe, and all such categories are assumed to be small with respect
to a fixed larger Grothendieck universe, cf. \cite[1.4]{TT}.)
We let $\Dbdg(X)$ denote the dg-quotient of $\Cohdg^b(X)$ by the full subcategory of exact complexes. This is a dg-category with the same objects as $\Cohdg^b(X)$
in which a contracting homotopy for each exact complex has been formally adjoined; see \cite{drinfeld} for the precise definition.
For any closed subset $Z \subseteq X$, we let $\DbdgZ(X)$ denote the dg-subcategory of $\Dbdg(X)$ given by complexes with support in $Z$.  
The dg-category $\DbdgZ(X)$ is pre-triangulated, and its associated homotopy category, which we shall write as $\DbZ(X)$, is
isomorphic to the usual bounded derived category of coherent sheaves on $X$ supported on $Z$.

\subsubsection{Mixed Hochschild complexes}
\label{secmixed}
We recall that a \emph{mixed complex of $k$-vector spaces} is a dg-module over the graded commutative $k$-algebra $k\langle e \rangle = k[e]/(e^2)$, where $|e| = -1$.
Typically, a mixed complex is thought of as a  triple $M = (M, b, B)$, where $(M, b)$ is a chain complex, and $B$ denotes the action of $e$; so $b^2 = 0$, $B^2 = 0$, and $[b,B] =
bB + Bb = 0$.  A  \emph{morphism of mixed complexes} refers to  a morphism of dg-$k\langle e \rangle$-modules. Such a morphism is a 
quasi-isomorphism if and only if it so upon forgetting the action of $B$. 
One may associate to any dg-category $\cC$ over $k$ a mixed complex $\MC(\cC)$, its \emph{mixed Hochschild complex}. We refer the reader to e.g. \cite[Section 3]{BW} for a detailed
discussion of mixed Hochschild complexes associated to dg-categories.

The periodic cyclic homology of a mixed complex $M$ is given by the ``Tate construction''.
In detail: the \emph{negative cyclic complex} associated to $M$ is $HN(M) \coloneqq \RHom_{k \langle e  \rangle}(k, M)$.
Since $HN(k) = k[u]$ for a degree two element $u$, $HN(M)$ is naturally a dg-$k[u]$-module, and the \emph{periodic cyclic complex} of $M$ is
$HP(M) \coloneqq HN(M) \otimes_{k[u]} k[u, u^{-1}]$. The \emph{periodic cyclic homology} of $M$ is defined to be $HP_*(M) \coloneqq H_*(HP(M))$. The periodic cyclic complex (resp. homology) of a dg-category $\cC$ over $k$ is given by $HP(\cC) \coloneqq HP(\MC(\cC))$ (resp. $HP_*(\cC) \coloneqq HP_*(\MC(\cC))$.
The assignment $\cC \mapsto HP(\cC)$ is covariantly functorial for dg-functors and sends quasi-equivalences of dg-categories to quasi-isomorphisms. 
See \cite[Section 3.2]{BW} for more details.

\subsection{Mayer-Vietoris for the Hochschild mixed complex}
In this subsection, we recall some background on localizing invariants of dg-categories. The localizing invariants of interest in this paper are the various Hochschild invariants discussed in Section~\ref{secmixed} (see Theorem~\ref{kellerthm}, due to Keller) and topological $K$-theory (see Theorem~\ref{KtopThm}, due to Blanc). The main goal of this subsection is to prove Corollary~\ref{MVcor}, a Mayer-Vietoris result for localizing invariants.

Let us fix a bit more notation/terminology. Given a dg-category $\cC$, we let $[\cC]$ denote its homotopy category. We say an object $X$ in
$\cC$ is {\em contractible} if $X$ is the zero object in $[\cC]$, or, equivalently, if the dga
$\End_{\cC}(X)$ is exact. Let $\cC_\ctr$ denote the full dg-subcategory of $\cC$ given by the contractible objects.

In this paper, a {\em short exact sequence} of dg-categories, written
$$
\cA \xra{\iota} \cB \xra{F} \cC,
$$
consists of pre-triangulated dg-categories $\cA$, $\cB$ and $\cC$ and  a dg-functor $F: \cB \to \cC$,  such that $\cA$ is a full dg-subcategory  of $\cB$ (with $\iota$ denoting the
inclusion functor), $F(A) \in   \cC_\ctr$ for all $A \in \cA$, and the triangulated functor induced by $F$ from the Verdier quotient $[\cB]/[\cA]$ to $[\cC]$ is
an equivalence.

We say a
commutative square 
$$
  \xymatrix{
    X \ar[r] \ar[d] & Y \ar[d] \\
    Z \ar[r] & W
  }
  $$
  of dg-modules over some dga  is {\em homotopy cartesian} if the following equivalent conditions hold:
  (1) its totalization is exact, (2) the induced map on the mapping cones of its rows   is a quasi-isomorphism, or (3) 
the induced map on the mapping cones of its columns   is a quasi-isomorphism.

We will make use of the following localization sequence of dg-categories, the essence of which is due to a result of Gabriel \cite[Ch. V]{gabriel}:

\begin{prop}
\label{gabrieldg}
Let $X$ be a noetherian scheme, $Y$ and $Z$ closed subschemes of $X$, and $U = X \setminus Z$. The sequence
\begin{equation}
\label{sequence}
 \D^{\on{b}, Y \cap Z}_{\on{dg}}(X) \to \DbdgY(X) \to \D^{\on{b}, U \cap Y}_{\on{dg}}(U),
\end{equation}
where the second functor is given by pullback along the open immersion $U \into X$,  is a short exact sequence of dg-categories. 
\end{prop}

\begin{proof}
  It suffices to show that the sequence
$$
 \D^{\on{b}, Y \cap Z}(X) \to \D^{\on{b}, Y}(X) \to \D^{\on{b}, U \cap Y}(U)
$$
of triangulated categories exhibits $\D^{\on{b}, U \cap Y}(U)$ as the Verdier quotient $\D^{\on{b}, Y}(X) /\D^{\on{b}, Y \cap Z}(X)$. The proof in \cite[Section 2.3.8]{schlichting}
that
$$
\coh^Z(X) \to \coh(X) \to \coh(U)
$$
is a short exact sequence of abelian categories extends verbatim to give a proof that
$$
\coh^{Y \cap Z}(X) \to \coh^Y(X) \to \coh^{U \cap Y}(U)
$$
is a short exact sequence of abelian categories. Now apply \cite[Lemma 4.4.1]{krause}. 
\end{proof}

The following is a slight modification of a notion found in, for example, \cite[Definition 4.3]{TVdB}:

\begin{defn} \label{localizing} Let $E$ be a functor from the category of small\footnote{By ``small", we mean small with respect to our above choices of Grothendieck
    universes.}   dg-categories over $k$ to the category of dg-modules over some fixed dga $\Lambda$.
  We say $E$ is {\em localizing} if the following two conditions hold:
  \begin{enumerate}
  \item If $G: \cA \to \cB$ is a dg-functor that is a Morita equivalence (e.g., a quasi-equivalence), then $E(G): E(\cA) \to E(\cB)$ is a quasi-isomorphism of dg-$\Lambda$-modules.
  \item  If $\cA \xra{\iota} \cB \xra{F} \cC$ is a short exact sequence of pre-triangulated dg-categories, then the   commutative square of dg-$\Lambda$-modules
  \begin{equation} \label{LocalizingE1}
  \xymatrix{
    E(\cA) \ar[r] \ar[d] & E(\cB) \ar[d] \\
    E(\cC_\ctr) \ar[r] & E(\cC)
  }
\end{equation}
\end{enumerate}
  is homotopy cartesian.
\end{defn}

\begin{thm}[Keller]
\label{kellerthm}
The functors $MC$, $HH$, $HN$, and $HP$ are localizing functors in the
sense of Definition~\ref{localizing},  taking values in mixed complexes over $k$, complexes over $k$, dg-$k[u]$-modules, and dg-$k[u, u^{-1}]$-modules, respectively.
\end{thm}

\begin{proof} The functor
$MC$ inverts Morita equivalences by \cite[Section 1.5]{kellercyclic}. Since the canonical map $[\cC] \to  [\cC]/[\cC_\ctr]$ is a quasi-equivalence, the Theorem in \cite[Section 2.4]{kellercyclic}
  implies that the induced map from the cone of the top arrow to the cone of the bottom arrow in \eqref{LocalizingE1} with $E = MC$ is a quasi-isomorphism.
This proves the result for $MC$. The result for the other three theories follows, since each is obtained from $MC$ by applying an additive  functor that preserves quasi-isomorphisms.
  \end{proof}

  \begin{thm}[Blanc]\label{KtopThm}
  The functor $K^{\on{top}}_\C$ from dg-categories over $\C$ to chain complexes over $\C$, given by sending a dg-category to its complexified topological $K$-theory,
  is localizing.
\end{thm}

\begin{proof}
This essentially follows from \cite[Proposition 4.15]{blanc}. In more detail: $K^{\on{top}}_\C$ inverts Morita equivalences by \cite[Proposition 4.15(b)]{blanc}. Nonconnective algebraic $K$-theory is a localizing invariant \cite[Theorem 9]{schlichting}; so it suffices to observe, as in the proof of \cite[Proposition 4.15(c)]{blanc}, that Blanc's topological realization functor $| - |_{\mathbb{S}}$, inverting the Bott element, and tensoring with $\C$ are all exact functors. 
\end{proof}

The following three Lemmas follow from straightforward diagram chases; we include a proof of the third and omit proofs of the first two. 

\begin{lem} \label{FiberLem}
  If $\alpha: E \to E'$ is a natural transformation of localizing invariants taking values in dg-$\Lambda$-modules, then
  so is the fiber of $\alpha$, written $\fiber(\alpha)$ and defined by 
  $$
  \cA \mapsto \cone(E(\cA) \xra{\alpha(\cA)} E'(\cA))[-1].
  $$
  In particular, the fiber of the complexified topological Chern character map $ch: K^{\on{top}}_\C \to HP$ is a localizing invariant.
  \end{lem}


\begin{lem} \label{DistLem} If $E$ is a localizing invariant taking values in dg-$\Lambda$-modules, then for every short exact sequence
  of pre-triangulated dg-categories $\cA \xra{\iota} \cB \xra{F} \cC$,  there is a distinguished triangle in $D(\Lambda)$, the derived category of dg-$\Lambda$-modules, of the form
  $$
  E(\cA) \to E(\cB) \to E(\cC) \xra{\del_{\cA,F}}  E(\cA)[1],
  $$
  where the map $\del_{\cA,F}$ is the composition
  $$
  E(\cC) \xra{\can_1} \cone\left(E(\cC_\ctr) \to E(\cC)\right) \xra{\alpha^{-1}}
\cone\left(E(\cA) \to E(\cB)\right)  \xra{\can_2}
  E(\cA)[1].
  $$
Here, $\alpha^{-1}$ is the inverse in $D(\Lambda)$ of the quasi-isomorphism $\alpha$ induced by \eqref{LocalizingE1}, and the maps $\can_1$ and $\can_2$ are the canonical maps to and from the mapping cone, respectively. 
\end{lem}


\begin{lem}  \label{NatDistLem}
Suppose $E$ is a localizing invariant (in the sense of Definition~\ref{localizing}) taking values in dg-$\Lambda$-modules, and each row of the commutative diagram
  $$
  \xymatrix{
    \cA \ar[r] \ar[d] & \cB \ar[r]^F \ar[d] & \cC \ar[d] \\
    \cA' \ar[r] & \cB' \ar[r]^{F'} & \cC' \\
  }
  $$
  is a short exact sequence of pre-triangulated dg-categories.
\begin{enumerate}
\item The vertical maps induce a morphism of distinquished triangles in $D(\Lambda)$  of the form
  \begin{equation} \label{E42}
  \xymatrix{
    E(\cA) \ar[r] \ar[d] & E(\cB) \ar[r] \ar[d] & E(\cC) \ar[r]^-{\del_{\cA,F}} \ar[d] &  E(\cA)[1]\ar[d] \\
      E(\cA') \ar[r]  & E(\cB') \ar[r] & E(\cC') \ar[r]^-{\del_{\cA',F'}}  &  E(\cA')[1],\\
}
  \end{equation}
  where the boundary maps are as defined in the statement of Lemma~\ref{DistLem}.

  \item If, in addition, the map
  $E(\cA) \xra{\simeq} E(\cA')$ is a quasi-isomorphism, then the commutative square of dg-$\Lambda$-modules
  \begin{equation} \label{E2}
  \xymatrix{
    E(\cB) \ar[r] \ar[d] & E(\cC) \ar[d] \\
    E(\cB') \ar[r] & E(\cC')
  }
 \end{equation}
  is homotopy cartesian.
\end{enumerate}
\end{lem}

\begin{proof}
  Both parts will involve the cube of dg-$\Lambda$-modules
\begin{equation} \label{cube}
\xymatrix{
E(\cA') \ar[rrr] \ar[ddd] &&& E(\cB') \ar[ddd] \\
 &E(\cA) \ar[r] \ar[d] \ar[lu] & E(\cB)  \ar[d] \ar[ru] &\\
&E(\cC_\ctr)  \ar[r] \ar[ld]  & E(\cC) \ar[rd] &  \\
E(\cC'_\ctr)  \ar[rrr] &&&  E(\cC'),
}
\end{equation}
which is commutative due to the functorality of $E$. Let us now prove (1). The left two squares of \eqref{E42} clearly commute; as for the right-most square in \eqref{E42}:
from the definition of the boundary map in Corollary \ref{DistLem}, we see that it suffices to show
  $$
  \xymatrix{
    \cone\left(E(\cA) \to E(\cB)\right) \ar[r]^{\alpha} \ar[d] & \cone\left(E(\cC_\ctr) \to E(\cC)\right) \ar[d] \\
        \cone\left(E(\cA') \to E(\cB')\right) \ar[r]^{\alpha'}  & \cone\left(E(\cC'_\ctr) \to E(\cC')\right)  \\
      }
      $$
      commutes. This is a consequence of the commutativity of \eqref{cube}.

To prove (2), let $I$, $O$, $L$, and $R$ denote the totalizations of the four commutative squares of dg-$\Lambda$-modules given by the
inner square, the outer square, the left-hand trapezoid, and the right-hand trapezoid of \eqref{cube}, respectively.
The commutativity of \eqref{cube} gives induced maps $I \to O$ and $L \to R$. Moreover, 
since both $\cone(I \to O)$ and $\cone(L \to R)$ are isomorphic to the totalization of \eqref{cube} regarded as a three-dimensional complex of dg-$\Lambda$-modules,
there is  an isomorphism  $\cone(I \to O) \cong \cone(L \to R)$.
Both $I$ and $O$ are exact since $E$ is localizing, and $L$ is exact since the top and bottom 
edges of the left-hand trapezoid are both  quasi-isomorphisms, the top one  by assumption and the bottom one since $E(\cC_\ctr)$ and $E(\cC'_\ctr)$ are exact.
It follows that $R$ is exact. 
\end{proof}

\begin{notation} \label{formaldef} Let $E$ be any functor from small dg-categories over $k$ to dg-modules over some dga $\Lambda$.
  For any noetherian $k$-scheme $X$ and closed subscheme $Y$ of $X$, we set
  $$
  E^Y(X) \coloneqq E(\Perfdg^Y(X)), \and   E_{\coh}^Y(X) \coloneqq E(\DbdgY(X)).
  $$
\end{notation}

\begin{cor}[Mayer-Vietoris]
\label{MVcor}
Let $X$ be a noetherian $k$-scheme, and suppose $X = U \cup V$, where $U$ and $V$ are open subschemes of $X$. Let $Y$ be a closed subscheme of $X$ 
and $E$ any localizing invariant taking values in dg-$\Lambda$-modules. The square
$$
\xymatrix{
E^Y_{\coh}(X) \ar[r] \ar[d] & E_{\coh}^{U \cap Y}(U)  \ar[d] \\
E_{\coh}^{V \cap Y}(V) \ar[r] & E_{\coh}^{U \cap V \cap Y}(U \cap V), \\
}
$$
in which each map is induced by pullback along an open immersion, is homotopy cartesian.
\end{cor}

\begin{proof}
Set $Z \coloneqq X \setminus U$ and $W \coloneqq V \setminus (U \cap V)$.  It follows from Proposition~\ref{gabrieldg}   that  
$$
\D^{\on{b}, \emptyset}_{\on{dg}}(X) \to \D^{\on{b}, Y \cap Z}_{\on{dg}}(X) \to \D^{\on{b}, W \cap Y}_{\on{dg}}(V)
$$
is a short exact sequence of dg-categories; since $\D^{\on{b}, \emptyset}_{\on{dg}}(X)$ has a trivial homotopy category,
we conclude that $\D^{\on{b}, Y \cap Z}_{\on{dg}}(X) \to \D^{\on{b}, W \cap Y}_{\on{dg}}(V)$ is a quasi-equivalence. Now apply Lemma~\ref{NatDistLem} to the
commutative diagram 
$$
\xymatrix{
\D^{\on{b}, Y \cap Z}_{\on{dg}}(X) \ar[r] \ar[d]  & \D^{\on{b}, Y}_{\on{dg}}(X) \ar[r] \ar[d] & \D^{\on{b}, U \cap Y}_{\on{dg}}(U \ar[d]) \\
\D^{\on{b}, W \cap Y}_{\on{dg}}(V) \ar[r] & \D^{\on{b}, V \cap Y}_{\on{dg}}(V) \ar[r] & \D^{\on{b}, U \cap V \cap Y}_{\on{dg}}(U \cap V). \\
}
$$
\end{proof}


\subsection{Koszul duality}
\label{koszulduality}

We recall in this section a Koszul duality statement that is essentially due to Martin \cite[Theorem 5.1]{martin}; see also work of Burke-Stevenson \cite[Theorem 7.5]{BS}. Let $Q$ be an essentially smooth algebra over a field $k$ and $f_1, \dots, f_c$ a (not necessarily regular) sequence of elements in $Q$. Let $\wQ = Q[t_1, \dots, t_c]$, where $|t_i| = 2$ for all $i$, and $\wf = f_1t_1 +  \cdots + f_ct_c \in \wQ$.

\begin{defn}
A \emph{matrix factorization of $\wf$} is a projective, finitely generated, $\Z$-graded $\wQ$-module $P$ equipped with a degree 1 endomorphism $d_P$ such that $d_P^2 = \wf \cdot \id_P$. Given two matrix factorizations $P$ and $P'$, we have a morphism complex $\Hom(P, P')$ with underlying graded module given by the internal $\Hom$ object $\underline{\Hom}_{\wQ}(P, P')$ in the category of $\Z$-graded $\wQ$-modules and differential given by $\a \mapsto d_{P'} \a - (-1)^{|\a|}\a d_{P}$. Let $mf(\wQ, \wf)$ denote the differential $\Z$-graded category with objects given by matrix factorizations of $\wf$ and morphism complexes given as above. 
\end{defn}

Matrix factorizations were introduced by Eisenbud in his study of the asymptotic behavior of free resolutions over local hypersurface rings \cite{eisenbud}. Since their inception in commutative algebra, matrix factorizations have appeared in a wide variety of branches of mathematics: for instance, homological mirror symmetry \cite{BHLW,HPSV, sheridan}, $K$-theory~\cite{brown16,BMTW2,lurie, walker17}, knot theory~\cite{KR08, KR082, oblomkov}, and noncommutative Hodge theory~\cite{HLP, KKP, PV}, among others.

Let $K$ denote the Koszul complex on $f_1, \dots, f_c$. The underlying $Q$-module of $K$ is $\bigwedge_Q(e_1, \dots, e_c)$, where each $e_i$ is an exterior variable of degree
$-1$. Let $\Dbdg(K)$ be the dg-quotient of the dg-category of finitely generated dg-$K$-modules by the subcategory of exact ones, 
as in \ref{enhancement}; and let $\cK$ be the dg-subcategory of $\Dbdg(K)$ on those dg-$K$-modules that are
projective as $Q$-modules. Notice that the inclusion $\cK \into \Dbdg(K)$ is a quasi-equivalence. 

Recall that $mf(\wQ, \wf)_\ctr$ is the dg-subcategory of $mf(\wQ, \wf)$ given by contractible objects. Let
$
\Phi: \cK \to mf(\wQ, \wf)/mf(\wQ, \wf)_\ctr
$
denote the dg-functor that sends an object $(P, d) \in \cK$ to the matrix factorization $(P[t_1, \dots, t_c], d + \sum_{i = 1}^c e_it_i)$; it follows from (a slight reformulation of) a result of Martin \cite[Theorem 5.1]{martin} that $\Phi$ is well-defined and is a quasi-equivalence. As observed in \cite{martin}, the functor $\Phi$ is an instance of Koszul duality. Indeed, when $Q = k$ and each $f_i = 0$, the equivalence $\Phi$ recovers (a nonstandard-graded variant of) the classical Bernstein-Gel'fand-Gel'fand correspondence between an exterior and polynomial algebra \cite{BGG}.

The following result, which plays a key role in the proof of Theorem \ref{devissage}, is now immediate:

\begin{prop}[\cite{martin}]
\label{koszulmf}
We have a commutative diagram of the form
$$
\xymatrix{
\Dbdg(K) \ar[rd] & \ar[l]_-\simeq \cK \ar[d] \ar[r]^-\simeq_-{\Phi} & mf(\wQ, \wf)/mf(\wQ, \wf)_\ctr          \ar[ld] & \ar[l]_-\simeq mf(\wQ, \wf)  \ar[lld] \\
& \DbdgZ(Q),\\
}
$$
where each horizontal functor is a quasi-equivalence. The left-most diagonal map and vertical map are forgetful functors, and the two right-most diagonal maps are given by setting each $t_i$ to 0. The left-most horizontal map is the inclusion, and the right-most horizontal map is the canonical one. 
\end{prop}

\subsection{An HKR-type theorem}

We have the following Hochschild-Kostant-Rosenberg (HKR)-type formula due to the second author, building on results of \cite{CT, PP, segal}:
\begin{thm}
\label{HKR}
Let $\wQ$ and $\wf$ be as in Section \ref{koszulduality}, and assume that $\on{char}(k) = 0$. There is a natural HKR-type isomorphism
$$
\MC(mf(\wQ, \wf)) \xra{\simeq} (\Omega^\bu_{\wQ/k},  d_{\wQ}\wf,d_{\wQ})
$$
in the derived category of mixed complexes, where $d_{\wQ} \wf$ denotes the map given by exterior multiplication on the left by the element $d_{\wQ} \wf \in \Omega^1_{\wQ/k}$. 
\end{thm}
\begin{proof}
It follows from a result of Efimov \cite[Proposition 3.14]{efimov2}
that there is a quasi-isomorphism
$$
\MC^{II}(\wQ, -\wf) \xra{\simeq} (\Omega^\bu_{\wQ/k},  d_{\wQ} \wf, d_{\wQ}),
$$
where $\MC^{II}(\wQ, -\wf)$ denotes the mixed Hochschild complex of the second kind of the curved algebra $(\wQ, -\wf)$; see e.g. \cite[Sections 2 and 3]{BW} for background on curved algebras and their Hochschild invariants of the second kind. 
By work of Polishchuk-Positselski \cite{PP}, there is a canonical isomorphism $\MC^{II}(\wQ, -\wf) \cong \MC^{II}(mf(\wQ, \wf))$ in the derived category of mixed complexes. See \cite[Proposition 3.25]{BW} for an explicit formulation of this result; note that the category $\Perf(\wQ,\wf)^{\op}$ in that statement coincides with $mf(\wQ, -\wf)^{\op} \cong mf(\wQ, \wf)$. Finally, by a result of the second-named author \cite{walker}, the canonical map $\MC(mf(\wQ, \wf)) \to \MC^{II}(mf(\wQ, \wf))$ is a quasi-isomorphism. 
\end{proof}

Combining Theorem \ref{HKR} with the horizontal quasi-equivalences in the diagram in Proposition~\ref{koszulmf}, one arrives at a formula for the mixed Hochschild complex of $\Dbdg(K)$. 

%


\section{Key technical result}
We begin by fixing the following  

\begin{notation} \label{notation}
  Let $S = \bigoplus_{j \geq 0} S^j$ be an $\N$-graded $k$-algebra essentially of finite type
  that is concentrated in even degrees (i.e., $S^j = 0$ for $j$ odd) and commutative.
  Given a degree two element $h \in S^2$, we define
  $$
  \begin{aligned}
HN^{dR}(S, h) & \coloneqq (\O^\bu_{S/k}[u], d_Sh+ ud_S) \\
HP^{dR}(S, h) & \coloneqq HN^{dR}(S,h) \otimes_{k[u]} k[u, u^{-1}] = (\O^\bu_{S/k}[u,u^{-1}], d_Sh+ ud_S), \\
\end{aligned}
$$
where $d_S$ denotes the de Rham differential, $u$ is a degree $2$ variable, and the summand $d_Sh$ of the differential indicates exterior multiplication on the left by the element
$d_Sh \in \O^1_{S/k}$.
In these complexes, the degree of an element $a_0 da_1 \cdots da_j \in \Omega^j_{S/k}$ is declared to be  $-j + \sum_i |a_i|$; in particular, the operator $d_S$ has degree $-1$
and both $d_Sh$ and $ud_S$ have degree $1$.

The symbols $HN^{dR}$ and $HP^{dR}$ are meant to indicate that, under certain conditions, these complexes are de Rham models of the negative cyclic and
periodic cyclic complexes of the matrix factorization category $mf(S, h)$: see Theorem \ref{HKR}, and also \cite{walker}.
\end{notation}

Let $A$ satisfy the assumption on $S$ in  \ref{notation}, and assume also that $A$ is essentially smooth over $k$.
Fix $f \in A^0$ and $g \in A^2$, and let $t$ be a degree $2$ variable. The goal of this section is to prove the following key technical result, which plays a crucial role in the proof of Theorem~\ref{devissage}.

\begin{prop}
\label{keylemma}
The square
\begin{equation}
\label{keysquare}
\xymatrix{
HP^{dR}(A[t], ft + g) \ar[r] \ar[d] & HP^{dR}(A, g) \ar[d] \\
HP^{dR}(A[1/f, t], ft + g) \ar[r] &   HP^{dR}(A[1/f],g ),
}
\end{equation}
in which the vertical maps are induced by inverting $f$  and the horizontal maps are induced by setting $t = 0$, is homotopy cartesian. Moreover, the bottom-left complex $HP^{dR}(A[1/f, t], ft + g)$ is $k[u]$-linearly contractible.
\end{prop}

Before proving Proposition \ref{keylemma}, we establish a series of intermediate technical results. Define a complex
$$
M = M_{A, f, g} =  (\O^\bu_{A/k}[t,u], tdf + dg + ud),
$$
where $d$ denotes the de Rham differential in $\O^\bu_{A/k}$. Here, as above, the summands $tdf$ and $dg$ of the differential denote exterior multiplication on the left by the
elements $tdf$ and $dg$ of $\O^1_{A/k}[t]$. (To clarify, if $g = g_0 + g_1t +\cdots + g_mt^m$, then $dg = dg_0 + tdg_1 + \cdots + t^m dg_m$.) 
Define
$$
\n: M \to M
$$
to be the $\O^\bu_{A/k}[u]$-linear chain endomorphism such that $\n(t^i) = ft^i +   it^{i - 1}u$; that is, $\n = f + u\frac{\del}{\del t}$. 

\begin{lem}
\label{cone}
There is an isomorphism
$$
HN^{dR}(A[t], ft + g) \cong  \fiber(\n) 
$$
of complexes of $A[t,u]$-modules, where the left-hand side is defined in Notation \ref{notation} using $S = A[t]$ and $h  = ft +g$, and $\fiber(\n) : = \on{cone}(\n)[-1]$.
\end{lem}

\begin{proof}
The composition 
$$
\O^\bu_{A[t]/k}[u]  \cong\O^\bu_{A/k}[t,u] \oplus  \O^\bu_{A/k}[t,u]dt  \cong\O^\bu_{A/k}[t,u] \oplus  (\O^\bu_{A/k}[t,u])[-1]
$$
gives the desired chain isomorphism.
\end{proof}

Define an $\O^\bu_{A/k}[u]$-linear map
$
\varphi : M \to HN^{dR}(A[1/f],g )
$
by $\varphi(t^i) = \frac{(-1)^{i+1} i! }{f^{i+1}}u^i$.

\begin{lem}
\label{philemma}
The map $\varphi$ has the following properties:
\begin{enumerate}
\item $\varphi$ is a chain map.
\item  $\varphi \circ \n = 0$.
\item The sequence
  $
 \left(0 \to M[-2]  \xra{\n \cdot t} M \xra{\varphi} HN^{dR}(A[1/f],g )\right)
  $
  is exact.
\item The sequence
  $
  \left(0 \to (M[u^{-1}])[-2]    \xra{\n \cdot t} M[u^{-1}]  \xra{\varphi} HP^{dR}(A[1/f],g ) \to 0\right)
  $
  is exact.
\end{enumerate}
\end{lem}

\begin{proof}
  Parts (1) and (2) are straightforward to check. Part (4) follows from (3), using that
  $\varphi$ becomes surjective upon inverting $u$. As for (3), one easily checks that $\n \cdot t$ is injective. Suppose $m \in \ker(\varphi)$. We construct forms $\b_{i, j} \in \Omega^\bu_{A/k}$ such that $(\n \cdot t)( \sum_{i \ge 0, j \ge 0} \beta_{i, j}t^iu^j) = m$. Write $m = \sum_{i,j \ge 0} \om_{i, j} t^iu^j$. We have
$$
\sum_{i + j = n}  \frac{(-1)^{i + 1}i! \om_{i,j}}{f^{i + 1}} = 0 \quad \text{for all $n \ge 0$.}
$$
In particular, $f$ divides $\om_{i, 0}$ for all $i \ge 0$. Set $\beta_{0,j} = \om_{0,j+1}$ and $\beta_{i,0} = \frac{\om_{i+1, 0}}{f}$. We define the forms $\beta_{i, j}$ for $i,j \ge 1$ inductively, on $i$, via the formula
$$
\beta_{i, j} = \frac{\om_{i, j+1} - f \b_{i-1, j+1}}{i+1}.
$$
Let $\widetilde{m} = \sum_{i \ge 0, j \ge 0} \beta_{i, j}t^iu^j$. Directly applying the formula for $\n \cdot t$, we have:
$$
(\n \cdot t)(\widetilde{m}) = \left(\sum_{i \ge 0} f \b_{i,0} t^{i+1}\right) + \left(\sum_{j \ge 0}  \b_{0,j} u^{j+1}\right) + \left(\sum_{i \ge 1, j \ge 1} (f \b_{i -1,j} + (i+1)\b_{i, j-1}  )t^iu^{j} \right).
$$
Now compare coefficients to check that $(\n \cdot t)(\widetilde{m}) = m $. \end{proof}

\begin{lem}
\label{contractible}
The complex $HN^{dR}(A[1/f, t], ft + g)$ is $k[u]$-linearly contractible, via the degree $-1$ $k[u]$-linear endomorphism $h$ of $HN^{dR}(A[1/f, t], ft + g)$ given by
$$
h(\om_1 t^i  + \om_2 t^j dt) = (-1)^{|\om_2|} \sum_{\ell} (-1)^\ell \frac{\om_2 u^\ell}{f^{\ell+1}} \frac{\partial^\ell (t^j)}{\partial t^\ell}
$$
for $\om_1, \om_2 \in \Omega_{A[1/f]}^\bu$. The complex $HP^{dR}(A[1/f, t], ft + g)$ is $k[u,u^{-1}]$-linearly contractible via a homotopy given by the same formula. 
\end{lem}

\begin{proof}
The second statement is immediate from the first, and the first follows from a direct calculation. 
\end{proof}

Consider the commutative square
\begin{equation}
\label{HNsquare}
\xymatrix{
HN^{dR}(A[t], ft + g) \ar[r] \ar[d]^-{\on{can}_1} & HN^{dR}(A, g) \ar[d]^-{\on{can}_2} \\
HN^{dR}(A[1/f, t], ft + g) \ar[r] &   HN^{dR}(A[1/f],g ),
}
\end{equation}
where $\on{can}_1$ and $\on{can}_2$ are the maps induced by inverting $f$,
and the horizontal maps are induced by setting $t = 0$. Note that
the square in Proposition \ref{keylemma} is obtained from \eqref{HNsquare} by inverting $u$.
The contracting homotopy $h$ of the bottom-left complex of \eqref{HNsquare} arising from Lemma \ref{contractible} induces the map $\sigma$
in the diagram
\begin{equation} \label{induced}
\xymatrix{
\fiber(\on{can}_1) \ar[r] \ar[d] & \fiber(\on{can}_2) \ar[d]\\
HN^{dR}(A[t], ft + g) \ar[r] \ar[d]^-{\on{can}_1} \ar[ru]^-{\sigma} & HN^{dR}(A, g) \ar[d]^-{\on{can}_2}  \\
HN^{dR}(A[1/f, t], ft + g) \ar[r] &   HN^{dR}(A[1/f],g ),}
\end{equation}
causing both triangles to commute. 

\begin{lem}
\label{anothertechlem}
We have a commutative diagram
$$
\xymatrix{
&&HN^{dR}_{A[t]/k}(A[t], ft + g) \ar[d]^-{\cong} \ar[dr]^-{\sigma}&\\
&& \on{fiber}(\nabla) \ar[d] \ar[r]& \fiber(\on{can}_2) \ar[d] \\
0 \ar[r] & M[-2] \ar[r]^-t \ar[d]^-{=}& M \ar[d]^{\nabla} \ar[r]^-{t = 0} & HN^{dR}(A, g)\ar[d]^-{\on{can_2}} \ar[r] & 0 \\
0 \ar[r] & M[-2] \ar[r]^-{\nabla \cdot t}  &  M\ar[r]^-{\varphi } & HN^{dR}(A[1/f],g ), \\
  }
$$
where the top-most vertical map is induced by Proposition \ref{cone}, and the bottom two rows are exact. 
\end{lem}

\begin{proof}
The exactness of the third row is clear, and the exactness of the fourth row follows from Lemma~\ref{philemma}(3). A direct calculation shows that the diagram commutes. 
\end{proof}

\begin{proof}[Proof of Proposition \ref{keylemma}]
The statement concerning the contractibility of $HP^{dR}(A[1/f, t], ft + g)$ follows from Lemma \ref{contractible}. Inverting $u$ in the diagram from Lemma \ref{anothertechlem} gives a commutative diagram 
$$
\xymatrix{
&&HP^{dR}_{A[t]/k}(A[t], ft + g) \ar[d]^-{\cong} \ar[dr]^-{\sigma}&\\
&& \on{fiber}(\nabla) \ar[d] \ar[r]& \fiber(\on{can}_2) \ar[d] \\
0 \ar[r] & (M[u^{-1}])[-2] \ar[r]^-t \ar[d]^-{=}& M[u^{-1}] \ar[d]^{\nabla} \ar[r]^-{t = 0} & HP^{dR}(A, g)\ar[d]^-{\on{can_2}} \ar[r] & 0 \\
0 \ar[r] & (M[u^{-1}])[-2] \ar[r]^-{\nabla \cdot t}  &  M[u^{-1}] \ar[r]^-{\varphi } & HP^{dR}(A[1/f],g ) \ar[r] & 0. \\
  }
$$
Notice that the bottom row is now a short exact sequence, by Lemma \ref{philemma}(4). A diagram chase shows that $\sigma$ is a quasi-isomorphism.

Finally, we consider the commutative diagram
\begin{equation} \label{bmap3}
\xymatrix{
\fiber(\on{can}_1) \ar[r]^\simeq \ar[d]^\simeq & \fiber(\on{can}_2) \ar[d]\\
HP^{dR}(A[t], ft + g) \ar[r] \ar[d]^-{\on{can}_1} \ar[ru]^-{\sigma}_{\simeq} & HP^{dR}(A, g) \ar[d]^-{\on{can}_2}  \\
HP^{dR}(A[1/f, t], ft + g) \ar[r] &   HP^{dR}(A[1/f],g )
}
\end{equation}
obtained from \eqref{induced} by inverting $u$. 
Since $HP^{dR}(A[1/f, t], ft + g)$ is contractible, the upper-left vertical map is a quasi-isomorphism as shown.
We just proved that $\sigma$ is a quasi-isomorphism, and thus so too is the top horizontal map.
This implies that the bottom square is homotopy cartesian.
\end{proof}

\section{Proof of Theorem \ref{devissage}}

We first address the affine case of Theorem \ref{devissage}. For convenience, we introduce the following
\begin{terminology}
We say a dg-functor $\cC \to \cD$ is an \emph{HP-equivalence} if it induces a quasi-isomorphism on periodic cyclic complexes. 
\end{terminology}

Let $k$ be a characteristic 0 field, $Q$ an essentially smooth $k$-algebra, $f_1, \dots, f_c \in Q$, and $Z = V(f_1, \dots, f_c) \subseteq \Spec(Q)$.
Let $K$ denote the Koszul complex on $f_1,\dots, f_c$, and set $R = Q/(f_1, \dots, f_c)$.
Note that the canonical ring map $Q \onto R$ factors as $Q \into K \onto R$;
these maps induce dg-functors $\Dbdg(R) \to \Dbdg(K) \to \DbdgZ(Q)$ given by restriction of scalars.

\begin{thm}
\label{affine}
With the notation just introduced:
\begin{enumerate}

\item The dg-functor $\Dbdg(K) \to \DbdgZ(Q)$ is an HP-equivalence. 

\item The dg-functor $\Dbdg(R) \to \Dbdg(K)$   is an HP-equivalence if and only if $\Dbdg(R) \to \DbdgZ(Q)$ is such.

 \item If $f_1, \dots, f_c$ form a regular sequence, then the dg functor $\Dbdg(R) \to \DbdgZ(Q)$ is an HP-equivalence.

\end{enumerate}
\end{thm}

\begin{proof}
  Part (2) follows immediately  from (1).  When $f_1, \dots, f_c$ is a regular sequence,
  the map $K \to R$ is a quasi-isomorphism and thus (3) follows from  (2).

 Let us now prove (1). Let $\wQ$ and $\wf$ be as in Subsection~\ref{koszulduality}. 
To prove (1), we argue by induction on $c$. Suppose $c = 1$, and write $f_1$ as just $f$. We have a commutative diagram \begin{equation}
\label{HPdiagram}
\xymatrix{
HP(mf(Q[t], ft)) \ar[r] \ar[d] &  HP(\DbdgZ(Q)) \ar[r] \ar[d] & HP(\Dbdg(Q)) \ar[d] \\
HP(mf(Q[1/f, t], ft))\ar[r] &  HP(\DbdgZ(Q[1/f]))\ar[r] & HP(\Dbdg(Q[1/f])).
}
\end{equation}
Each map in \eqref{HPdiagram} is induced by a dg-functor: the left-most horizontal maps are induced by setting $t = 0$, the right-most horizontal maps are induced by inclusions,
and the vertical maps are induced by inverting $f$. By Proposition \ref{koszulmf}, it suffices to show that the upper-left map is a quasi-isomorphism. 

We first observe that, by Theorem \ref{HKR}, the outer rectangle in \eqref{HPdiagram} is quasi-isomorphic to 
$$
\xymatrix{
HP^{dR}(Q[t], ft) \ar[r] \ar[d] & HP^{dR}(Q, 0) \ar[d] \\
HP^{dR}(Q[1/ft], ft) \ar[r] & HP^{dR}(Q[1/f], 0), \\
}
$$
which is homotopy cartesian by Proposition \ref{keylemma} (take $g = 0$ in that statement). The right-most square in \eqref{HPdiagram} is homotopy cartesian by Proposition~\ref{gabrieldg}, Theorem~\ref{kellerthm}, and the observation that $\DbdgZ(Q[1/f])$ is exactly the subcategory of contractible objects in $\Dbdg(Q[1/f])$. It follows that the left-most square in \eqref{HPdiagram} is also homotopy cartesian.  
The complex $HP(\DbdgZ(Q[1/f]))$ is exact since $\DbdgZ(Q[1/f])$ is quasi-equivalent to $0$,
and $HP(mf(Q[1/f, t], ft))$ is exact by Lemma \ref{contractible} and Theorem \ref{HKR}.
It follows that the top-left map in \eqref{HPdiagram} is a quasi-isomorphism; this proves the $c = 1$ case. 





Now suppose $c > 1$. For the same reasons as in the $c=1$ case, the complexes $HP(mf(\wQ[1/f_c], \wf))$ and $HP(\DbdgZ(Q[1/f_c]))$ are contractible.
It therefore suffices, by Proposition \ref{koszulmf}, to show that the square
\begin{equation}
\label{mainsquare}
\xymatrix{
HP(mf(\wQ, \wf)) \ar[r] \ar[d] & HP(\DbdgZ(Q)) \ar[d] \\
HP(mf(\wQ[1/f_c], \wf)) \ar[r] & HP(\DbdgZ(Q[1/f_c]))
}
\end{equation}
is homotopy cartesian. Let $\wQ' = Q[t_1, \dots, t_{c-1}]$, $\wf' = f_1t_1 + \cdots + f_{c-1}t_{c-1} \in \wQ'$, and $Z' = V(f_1, \dots, f_{c-1})$. We have the following commutative diagram
\begin{footnotesize}
$$
\xymatrix{
 HP(\DbdgZ(Q)) \ar[rrr] \ar[ddd] &&& HP(\D^{\on{b},Z'}_{\on{dg}}(Q)) \ar[ddd] \\
 &HP(mf(\wQ, \wf)) \ar[r] \ar[d] \ar[lu] & HP(mf(\wQ', \wf'))  \ar[d] \ar[ru] &\\
&HP(mf(\wQ[1/f_c], \wf))  \ar[r] \ar[ld] & HP(mf(\wQ'[1/f_c], \wf')) \ar[rd] &  \\
HP(\DbdgZ(Q[1/f_c]))  \ar[rrr] &&&  HP(\D^{\on{b},Z'}_{\on{dg}}(Q[1/f_c])),
}
$$
\end{footnotesize}

\noindent where the vertical maps are induced by inverting $f_c$,
the exterior horizontal maps are induced by inclusion, and every other map is given by sending one or more of the $t_i$ to 0. Observe that the square \eqref{mainsquare} is the
left-most trapezoid in this diagram. The diagonal arrows in the right-most trapezoid are quasi-equivalences by induction and Proposition \ref{koszulmf}; it follows that this
trapezoid is homotopy cartesian. The exterior square is homotopy cartesian by Proposition~\ref{gabrieldg}, Theorem~\ref{kellerthm}, and the observation that $\DbdgZ(Q[1/f_c])$ is exactly the subcategory of contractible objects in $\D^{\on{b},Z'}_{\on{dg}}(Q[1/f_c])$. The interior square is homotopy cartesian by Proposition \ref{keylemma}
and Theorem \ref{HKR}. By a diagram chase similar to the argument in the proof of Lemma~\ref{NatDistLem}(2), it follows that \eqref{mainsquare} is homotopy cartesian. This proves (1). 
\end{proof}

\begin{proof}[Proof of Theorem \ref{devissage}]
  Since $X$ is noetherian, we have $X = Y_1 \cup \cdots \cup Y_n$ with each $Y_i$ an affine open  subscheme of $X$ such that $Z \cap Y_i \into Y_i$ is lci. Each $Y_i$ is smooth
  since $X$ is. 

  We proceed by induction on $n$; the case $n = 1$ is the content of Theorem \ref{affine}(3). For $n \geq 2$, we have $X = U \cup V$, where we set $U \coloneqq Y_1 \cup \cdots \cup Y_{n-1}$ and $V \coloneqq Y_n$.
This gives a commutative diagram
\begin{footnotesize}
$$
\xymatrix{
 HP(\DbdgZ(X)) \ar[rrr] \ar[ddd] & & & HP(\DbdgU) \ar[ddd] \\
& HP(\Dbdg(Z)) \ar[r] \ar[d] \ar[lu] & HP(\Dbdg(Z \cap U)) \ar[d] \ar[ru]^\simeq & \\
& HP(\Dbdg(Z \cap V))  \ar[ld]_\simeq  \ar[r] & HP(\Dbdg(Z \cap U \cap V))\ar[rd]^\simeq  &  \\
 HP(\DbdgV) \ar[rrr] & & & HP(\DbdgUV). \\
}
$$
\end{footnotesize}

\noindent in which the diagonal maps are induced  by pushforward, and all other maps are induced by pullback.
By induction on $n$, the lower-left and upper-right diagonal maps are quasi-isomorphisms as indicated.
Observe that $U \cap V = (Y_1 \cap Y_n) \cup \cdots \cup (Y_{n-1} \cap Y_{n})$. Since $X$ is separated, each $Y_i \cap Y_n$ is affine, and the inclusion $Z \cap Y_i \cap Y_n \into  Y_i \cap Y_n$ is lci, for all $i$.
This proves the lower-right diagonal map is also a quasi-isomorphism as indicated. Finally, by Corollary \ref{MVcor}, the interior and exterior squares are homotopy cartesian. It follows that 
$HP(\Dbdg(Z)) \to HP(\DbdgZ(X))$ is a quasi-isomorphism.
\end{proof}

\begin{rem}
\label{tabuada}
When $Z$ is smooth, Theorem \ref{devissage} follows easily from the Hochschild-Kostant-Rosenberg Theorem and the Gysin long exact sequence in de Rham cohomology \cite[Section 2, Theorem 3.3]{hartshorne}; see also \cite[Example 1.15]{TVdB}. Since the Gysin sequence is not available when $Z$ is not smooth, this approach does not work in our setting. 

Similarly, the proof of Tabuada-Van den Bergh's result \cite[Theorem 1.8]{TVdB}, which states that devissage holds for localizing $\A^1$-homotopy invariants in the case of a closed embedding of a smooth scheme $Z$ into a smooth scheme $X$, does not extend to give a proof of Theorem~\ref{devissage}. One reason for this is that \cite[Theorem 6.8(ii)]{TVdB}, which plays a key role in the proof of \cite[Theorem 1.8]{TVdB}, does not extend to our setting. In more detail: \cite[Theorem 6.8(ii)]{TVdB} states that, if $R \to S$ is a surjective morphism of smooth $k$-algebras, then $\mathbb{R}\Hom_R(S, S)$ is a formal dga. To adapt Tabuada-Van den Bergh's argument to prove Theorem~\ref{devissage}, one would need a version of this result in the case where $S$ is assumed only to be a complete intersection. But this is simply false; for instance, when $R = k[x]$ and $S = k[x]/(x^2)$, it is straightforward to check that $\mathbb{R}\Hom_R(S, S)$ is not a formal dga. 
\end{rem}

\begin{rem}
Let $Q$, $R$, and $K$ be as in Theorem \ref{affine}. If we knew that the canonical map $HP(\Dbdg(R)) \to HP(\Dbdg(K))$ is a quasi-isomorphism for any (not necessarily regular) sequence $f_1, \dots, f_c \in Q$, the above Mayer-Vietoris argument would give a proof of Theorem \ref{devissage} without the lci assumption (but still assuming $X$ is smooth).
\end{rem}

Theorem \ref{devissage} admits a slight generalization:

\begin{cor}
\label{generalization}
Let $Z \into X$ and $X \into Y$ be closed embeddings. If d\'evissage for periodic cyclic homology holds for the embeddings $X \into Y$, $Z \into Y$, and $X \setminus Z \into Y
\setminus Z$; then it also holds for the embedding $Z \into X$. In particular, if $X \into Y$ is lci, $Z \into Y$ is lci, and $Y$ is smooth, then d\'evissage holds for $Z \into X$. 
\end{cor}

\begin{proof}
Consider the diagram
$$
\xymatrix{
  \Dbdg(Z) \ar[r] \ar@/^2.0pc/[rrr]^{\text{$HP$-eq}}
  & \D_{\on{dg}}^{\on{b}, Z}(X)\ar[rr] \ar[d] && \ar[d] \D_{\on{dg}}^{\on{b}, Z}(Y) \\
&\Dbdg(X) \ar[rr]^{\text{$HP$-eq}} \ar[d] && \D_{\on{dg}}^{\on{b}, X}(Y) \ar[d] \\
&\Dbdg(X \setminus Z) \ar[rr]^{\text{$HP$-eq}} && \D_{\on{dg}}^{\on{b}, X \setminus Z}(Y \setminus Z),
}
$$
in which all horiztonal maps are induced by pushforward, the two upper vertical maps are inclusions, and the bottom vertical maps are induced by 
pullback.
The curved arrow and  the bottow two  horizontal arrows  are $HP$-equivalences as indicated, by assumption.  Since the two columns are short exact sequences of
dg-categories by Proposition \ref{gabrieldg}, it follows from Theorem \ref{kellerthm} and Lemma \ref{NatDistLem} that the top-right horizontal arrow is also an $HP$-equivalence.
It follows that $\Dbdg(Z) \to \DbdgZ(X)$ is an HP-equivalence. The final assertion  follows by using Theorem \ref{devissage}. 
\end{proof}

\begin{ex}
  Suppose $X$ is a $k$-scheme that can be embedded via an lci closed embedding into a scheme that is smooth over $k$. 
  By Corollary \ref{generalization}, d\'evissage for periodic cyclic homology holds for any closed embedding $Z \into X$ provided $Z$ is smooth over $k$. (This follows from the Corollary
  since every closed embedding of smooth schemes is lci.)   For instance, d\'evissage holds for the inclusion of any point into $X$. 
\end{ex}

\section{The boundary map in a localization sequence on periodic cyclic homology}
Let $Q$ be an essentially smooth $k$-algebra, $f \in Q$
a non-zero-divisor, and $R = Q/f$. 

\subsection{Computing the boundary map}
Theorems~\ref{devissage} and~\ref{kellerthm} give a two-periodic long exact sequence 
\begin{equation} \label{bmap2}
\cdots \to HP_j(Q) \to  HP_j(Q[1/f]) \xra{\del_j} HP_{j-1}(\Dbdg(R)) \to HP_{j-1}(Q) \to \cdots.
\end{equation}
The goal of this subsection is to give an explicit formula for the boundary map $\del_j$. To achieve this, we use the de Rham versions of these complexes provided by Proposition~\ref{koszulmf} and Theorem \ref{HKR}, i.e. the isomorphisms 
\begin{equation} \label{bmap1}
  HP_*(\Dbdg(R)) \cong HP^{dR}_*(Q[t], ft) \and HP_*(Q[1/f]) \cong HP_*^{dR}(Q[1/f]),
\end{equation}
where the right-hand sides are defined as in Notation \ref{notation}.

\begin{lem} \label{lem:cycledescription} With $Q$ and $f$ as above, every class in $HP_j^{\dR}(Q[1/f])$ is represented by a sum of cycles of the form
  $\frac{\a}{f^s} u^{l}$  for $s,l \in \Z$ with $s \geq 0$ and  $\a \in \Omega^{2l+j}_Q$ satisfying $f d\a = s df \a$.
\end{lem}

\begin{proof}
We have $HP^{\dR}(Q[1/f]) = \bigoplus_{p,l}
\Omega_{Q[1/f]/k}^p u^l$, with $\Omega_{Q[1/f]/k}^p u^l$ in homological degree $p-2l$, and
differential  $u d$. There is an isomorphism
$
 \bigoplus_m H_{\dR}^{2m+j}(Q[1/f]) \cong HP^{\dR}_j(Q[1/f]),
$
where $H_{\dR}^*( - )$ refers to classical de Rham cohomology, that
sends the class of a closed form $\omega \in
\Omega^{2m+j}_{Q[1/f]}$ to the class of $\omega u^m$.  
Using the identification $\Omega^*_{Q[1/f]} \cong \Omega^*_Q[1/f]$, it
follows that a cycle in $HP^{\dR}(Q[1/f])$ of homological degree $j$  
is a finite sum of elements of the form $\frac{\a}{f^s} u^l$, with $\a \in \Omega^{2l+j}_Q$, each of which is a cycle satisfying $f d\a = s df \a$.
\end{proof}

\begin{rem}
\label{rem:technical}
The condition $fd\alpha = sdf\alpha$ in Lemma~\ref{lem:cycledescription} implies that $f df d\a = 0$, and hence, since $f$ is a non-zero-divisor, that $df d\a = 0$. 
\end{rem}

\begin{thm} \label{thm:boundarymap} Under the isomorphisms in \eqref{bmap1}, the boundary map $\del_j$ in \eqref{bmap2} corresponds to the map
  $
\del_j^{\dR}:  HP^{\dR}_j(Q[1/f]) \to HP^{\dR}_{j-1}(Q[t], ft)
  $
that sends a class $\frac{\a}{f^s} u^{l}$ as in Lemma \ref{lem:cycledescription} to
$
\frac{(-1)^{s}}{s!} d(\a t^s) u^{l+1-s}.
$
\end{thm}



\begin{proof} If $s = 0$, then this class lifts to an element of $HP_j^{dR}(Q)$ and hence is mapped to zero via $\del^{dR}$; henceforth, assume $s \geq 1$.  The element $\gamma = \gamma(s, \a) \coloneqq \frac{(-1)^{s}}{s!} d(\a t^s) u^{l+1-s} $ has degree $j-1$; let us check that $\gamma \in HP^{\dR}(Q[t], ft)$ is a cycle. We have $f d\a = s df \a$, and Remark~\ref{rem:technical} implies that $df d\a = 0$. We now compute:
\begin{align*}
  (u d + d(ft))(d\a t^s)  &= (-1)^{j+1} s d\a  t^{s-1} dt u + df d\a t^{s+1} + (-1)^{j+1}f d\a t^s dt \\
  &= (-1)^{j+1} \left( s d\a  t^{s-1} dt u + s df \a t^s dt\right). 
\end{align*}
Finally, to conclude that $\gamma$ is a cycle, we observe that $(ud + d(ft)) (s\a t^{s-1} dt) = s d\a t^{s-1} dt u + s df \a t^s dt$.

Consider diagram \eqref{bmap3} with $A = Q$ and $g = 0$. The complex $\fiber(\can_2)$ in that diagram is the graded $k[u, u^{-1}]$-module
$HP^{\dR}(Q) \oplus   HP^{\dR}(Q[1/f])[1]$ equipped with the differential $\begin{bmatrix} -d_{HP^{\dR}(Q)} & 0 \\ -\can_2 & d_{HP^{\dR}(Q[1/f])} \end{bmatrix}$,
and the boundary map $\del_j: HP^{\dR}_j(Q[1/f]) \to H_{j-1}(\fiber(\can_2))$ is induced by inclusion into the second summand. The map 
$\del^{\dR}_j$ is given by the composition $H_{j-1}(\sigma)^{-1}\del_j$, where $\sigma$ is as in diagram~\eqref{bmap3}. It therefore suffices to show that $H_{j-1}(\sigma)(\gamma) = \del_j\left(\frac{\a}{f^s}   u^{l}\right)$. 

Let $\tau: HP^{\dR}(Q[t], ft) \to HP^{\dR}(Q)$ be the map induced by setting $t = 0$. The quasi-isomorphism $\s$ is induced by a contracting homotopy $\overline{h}$ of $\can_2 \circ \tau: HP^{\dR}(Q[t], ft) \to H^{\dR}(Q[1/f])$;
specifically, $\s = \begin{bmatrix} \tau \\ -\overline{h} \end{bmatrix}$. 
The homotopy $\overline{h}$ is induced by the contracting homotopy $h$ given in Lemma \ref{contractible} and the bottom commutative square in \eqref{bmap3} and is thus given
by 
$$
\overline{h}\left(\omega_1 t^a + \omega_2 t^b dt\right) =
(-1)^{|\omega_2| +b} b! \frac{\omega_2 u^b}{f^{b+1}} 
$$
for $\omega_1, \omega_2 \in \Omega^\bu_{Q}$ and integers $a,b \geq 0$. 
It follows that 
$$
H_{j-1}(\s)(\g) =
\begin{bmatrix} \tau(\g) \\ -\overline{h}(\g) 
  \end{bmatrix} 
  =
  \begin{bmatrix} 0 \\ -\frac{(-1)^{s}}{s!} (-1)^j (-1)^{2l+j+s-1} (s-1)! \frac{s \a u^{s-1}}{f^s} u^{l+1-s} \end{bmatrix} \\
   = \del_{j}\left(\frac{\a}{f^s}   u^{l}\right).
$$
  \end{proof}


  \subsection{Relationship with Chern characters of matrix factorizations}
  In this subsection, we illustrate our explicit formula for the boundary map in~\eqref{bmap2} by showing it is compatible with the Chern character map for $K_1$.  For simplicity, we assume $Q$ is local
  (and essentially smooth over $k$) throughout this subsection. We will use Theorem~\ref{thm:boundarymap} to directly check that the square
    $$
    \xymatrix{
      K_1(Q[1/f]) \ar[r]^{\del_1} \ar[d]^{ch^{HP}_1} & G_0(Q/f) \ar[d]^{ch^{HP}_0} \\
      HP_1(Q[1/f]) \ar[r]^{\del_1} & HP_0(\Dbdg(Q/f)) \\
    }
    $$
    commutes, where $ch^{HP}_*$ denotes the $HP$-Chern character map, and the horizontal maps are the boundary maps in the canonical long exact sequences. 
    Using Proposition~\ref{koszulmf} and Theorem \ref{HKR}, we may identify this square with
    \begin{equation} \label{E923}
    \xymatrix{
      K_1(Q[1/f]) \ar[r]^{\del_1} \ar[d]^{ch^{\dR}_1} & G_0(Q/f) \ar[d]^{ch_0^{\dR}}  \\
      HP^{\dR}_1(Q[1/f]) \ar[r]^{\del_1^{\dR}} & HP^{\dR}_0(Q[t], ft), \\
    }
  \end{equation}
  where the maps $ch_*^{\dR}$ denote the de Rham versions of the Chern character maps $ch^{HP}_*$.
    Let us recall the formulas for the maps in this square.

  Since $Q$ and $Q[1/f]$ are regular, the long exact sequence in $G$-theory gives an exact sequence
  \begin{equation} \label{E912c}
  \cdots \to K_1(Q) \to K_1(Q[1/f]) \xra{\del_1} G_0(Q/f) \to K_0(Q) \to K_0(Q[1/f]) \to 0.
\end{equation}
The map $K_0(Q) \to K_0(Q[1/f])$ is injective, since $Q$ is local. Moreover, as $K_1(Q)$ is isomorphic to the group of units $Q^\times$ in $Q$, the boundary map induces an isomorphism $\frac{K_1(Q[1/f])}{Q^\times} \xra{\cong} G_0(Q/f)$.
The group $G_0(Q/f)$ is generated by the classes of maximal Cohen-Macaulay $Q/f$-modules. Given such a module $M$, it has projective dimension~$1$ as a $Q$-module, and thus there exists an exact sequence of the form
$$
0 \to Q^n \xra{A} Q^n \to M \to 0
$$
for some $n \times n$ matrix $A$ with entries in $Q$. Since multiplication by $f$ on $M$ is zero, there is a unique $n \times n$ matrix $B$ with entries in $Q$ such that $AB = BA = f \cdot I_n$; that is, $(A,B)$ forms a matrix factorization of
$f$. 
By \cite[Theorem III.3.2]{Kbook}, we have $\del_1([A]) = [\coker(A)] = [M]$. In particular, we need only check that the square~\eqref{E923} commutes on classes of the form $[A] \in K_1(Q[1/f])$, where $(A, B)$ is a matrix factorization of $f$. 

%



For any essentially smooth $k$-algebra $S$, the Chern character map
  $$
  ch_1^{\dR}: K_1(S) \to HP_1(S) \cong HP^{\dR}_1(S)
  $$
is given by
  $$
  ch_1^{\dR}(T) \coloneqq 
\sum_{s \geq 1} (-1)^{s+1}\frac{2 s!}{(2s)!} \tr(T^{-1} dT (dT^{-1} dT)^{s-1}) u^{s-1} 
$$
for any $T \in GL(S)$ \cite[Section 1]{pekonen}; here, we use the relation $(T^{-1}dT)^2 = -dT^{-1}dT$.\footnote{Our formula for $ch_1$ differs from the one found in \cite[Section 1]{pekonen} by the constant $\frac{i^{3s-2}}{(2\pi)^s}$.}
Applying this formula when $S = Q[1/f]$ and $T= A$,  where $(A,B)$ is a matrix factorization of $f \in Q$,
and using  \cite[Lemma 5.7]{BW1} along with the relation $dA^{-1} = f^{-1} dB - f^{-2} df B$, we obtain
\begin{equation} \label{E49b}
  ch^{\dR}_1(A) = \sum_{s \geq 1} (-1)^{s+1} \frac{2 s!}{(2s)!} f^{-s} \tr(B dA (dB dA)^{s-1}) u^{s-1} \in HP_1^{\dR}(Q[1/f]).
  \end{equation}
A similar calculation shows
  $$
f \tr((dB dA)^s) = s \cdot df \smsh \tr(B dA (dB dA)^{s-1})
$$
for each $s$. We now apply Theorem \ref{thm:boundarymap} to get
$$
 \del_1^{\dR}(ch_1(A)) = -\sum_{s \geq 1} \frac{2}{(2s)!}  \tr(d(BdA(dAdB)^{s-1}) t^s),
$$
which coincides with $ch^{\dR}_0([M])$ by \cite[Example 6.4]{BW}. This shows that \eqref{E923} commutes.

\section{Proof of Theorem \ref{latticeci}}

\begin{prop} \label{latticeprop}
Let $Z \into X$ be a closed embedding of $\C$-schemes, where $X$ is smooth.  
\begin{enumerate}
\item The Lattice Conjecture (Conjecture \ref{latticeconj}) holds for $\DbdgZ(X)$.
\item The following are equivalent:
\begin{enumerate}
\item The map $HP(\Dbdg(Z) )\to HP(\DbdgZ(X))$ induced by pushforward is a quasi-isomorphism.
\item The Lattice Conjecture holds for the dg-bounded derived category $\Dbdg(Z)$.
\item The Lattice Conjecture holds for the dg-singularity category $\D^{\sing}_{\on{dg}}(Z)$.
\end{enumerate}
\end{enumerate}
\end{prop}

\begin{proof}
  Let $E = \fiber(ch: K^{\on{top}}_\C \to HP)$, and note that the Lattice Conjecture holds for a dg-category $\cA$ if and only if $E(\cA)$ is exact.
Moreover, $E$ is localizing by Theorems \ref{kellerthm} and \ref{KtopThm}, Lemma \ref{FiberLem}, and the naturality of $ch$ \cite[Theorem 4.24]{blanc}.
In particular, the first assertion is equivalent to the assertion that $E_{\coh}^Z(X)$ is exact (see Notation~\ref{formaldef}). 
  Since $X$ and $X \setminus Z$ are smooth, and $E$ is localizing,  the canonical maps $E(X) \xra{\simeq} E_{\coh}(X)$ and
  $E(X \setminus Z) \xra{\simeq} E_{\coh}(X \setminus Z)$ are equivalences. Since the Lattice Conjecture is known for perfect complexes of separated schemes of finite type over
  $\C$,   we conclude that both $E_{\coh}(X)$ and $E_{\coh}(X \setminus Z)$ are exact. 
  The first assertion thus follows from Proposition \ref{gabrieldg} and Lemma \ref{DistLem}.

  As for (2), we recall that the map $K^{\on{top}}_\C(\Dbdg(Z)) \to K^{\on{top}}_\C(\DbdgZ(X))$ induced by pushforward along $Z \into X$
  is known to be an equivalence by \cite[Example 2.3]{HLP}. Using the naturality of $ch$, 
  it follows from (1) that (a) and (b) are equivalent.
By the definition of the dg-singularity category, we have a short exact sequence 
$
\Perfdg(Z) \to \Dbdg(Z) \to \D^{\sing}_{\on{dg}}(Z)
$
of pre-triangulated dg-categories. Since $E(Z)$ is exact, the equivalence of (b) and (c) follows from  Lemma \ref{DistLem}. 
  \end{proof}

\begin{proof}[Proof of Theorem \ref{latticeci}] 
  Let $E$ be the fiber of the Chern character map as in the proof of Proposition \ref{latticeprop}, so that the goal is to show $E(\Dbdg(X))$ and
$E(\D^{\sing}_{\on{dg}}(X))$ are exact. Since $X$ is noetherian, the assumptions give a  cover
  $X = U_1 \cup \cdots \cup U_n$ of $X$ by affine open subschemes such that each $U_i$ admits an lci embedding into a smooth $\C$-scheme.
By Theorem \ref{devissage} and   Proposition \ref{latticeprop}(2), $E_{\coh}(U_i)$ is exact for all $i$. 
Just as in the proof of Theorem \ref{devissage},  since $E$ is localizing, by induction on $n$ we conclude that $E_{\coh}(X)$ is exact.
Using  Proposition~\ref{latticeprop} again, we have that $E(\D^{\sing}_{\on{dg}}(X))$ is also  exact. 
\end{proof}

\begin{rem}
As discussed in the introduction, Khan has subsequently generalized Theorem~\ref{latticeci}: see \cite[Theorem B]{khan}. His result follows from a devissage statement \cite[Theorem A.2]{khan} by essentially the same argument as the one we give here. Additional new cases of the Lattice Conjecture for bounded derived categories and singularity categories of Gorenstein dg-algebras have also recently been obtained by the first author and Sridhar~\cite{BSlattice}.
\end{rem}

\bibliographystyle{amsalpha}
\bibliography{Bibliography}

\end{document}